\documentclass{article}
\usepackage{amsmath}
\usepackage{fancyhdr}

\newtheorem{theorem}{Theorem}

\newtheorem{definition}[theorem]{Definition}
\newtheorem{example}[theorem]{Example}

\newtheorem{lemma}[theorem]{Lemma}

\newtheorem{proposition}[theorem]{Proposition}
\newtheorem{remark}[theorem]{Remark}

\newenvironment{proof}[1][Proof]{\noindent\textbf{#1.} }{\ \rule{0.5em}{0.5em}}
\input{tcilatex}
\input{tcilatex}
\begin{document}

\date{\today }
\title{On a class of submanifolds in a tangent bundle with a $g$-natural
metric - normal lift}
\author{Stanis\l aw Ewert-Krzemieniewski}
\maketitle

\begin{abstract}
An isometric immersion of a Riemannian manifold $M$ into a Riemannian
manifold $N$ gives rise in a natural way to variety of immersions into the
tangent bundle $TN$ with a non-degenerate $g$-natural metric $G.$ In the
paper we introduce and study an immersion into $TN$ defined by the immersion 
$f:M\longrightarrow N$ itself and the normal bundle.


\end{abstract}

\section{Introduction}

An isometric immersion of a Riemannian manifold $M$ into a Riemannian
manifold $N$ gives rise in a natural way to variety of immersions into the
tangent bundle $TN$ with a non-degenerate $g$-natural metric $G.$ The
isometric immersion defined by the tangent bundle of the submanifold was
introduced by the author in \cite{arXiv:1411.3274}, \cite{Coll. Math.}. In
the present paper we introduce and study an immersion $\widetilde{f}%
:LM\longrightarrow TN$ defined by the immersion $f:M\longrightarrow N$
itself and the normal bundle.

In Preliminaries we recall basic facts on the decomposition of the tangent
bundle and $g$-natural metrics. We also present basic notions on
submanifolds and give short resum\'{e} on van der Waerden-Bortolotti
covariant derivative. In Section 3 basic equations are presented. The main
results are given in Section 4. We give the condition sufficient for $LM$
being totally geodesic submanifold of $TN.$

Throughout the paper all manifolds under consideration are Hausdorff and
smooth. The metrics on the base manifolds are Riemannian and the metrics on
tangent spaces are non-degenerate. We adopt the Einstein summation
convention.

\section{Preliminaries}

\subsection{Decomposition of the tangent space}

Let $\pi :TN\longrightarrow N$ be the tangent bundle of a Riemannian
manifold $N$ with the Levi-Civita connection $\nabla $ on $N,$ $\pi $ being
the projection. Then at each point $(x,u)\in TN$ the tangent space $%
T_{(x,u)}TN$ splits into direct sum of two isomorphic spaces $V_{(x,u)}TN$
and $H_{(x,u)}TN,$ where 
\begin{equation*}
V_{(x,u)}TN=Ker(d\pi |_{(x,u)}),\quad H_{(x,u)}TN=Ker(K|_{(x,u)})
\end{equation*}%
and $K$ is the connection map \cite{Dombr}, see also \cite{1}.

More precisely, if $Z=\left( Z^{r}\frac{\partial }{\partial x^{r}}+\overline{%
Z}^{r}\frac{\partial }{\partial u^{r}}\right) |_{(x,u)}\in T_{(x,u)}TN,$ $%
r=1,...,n,$ then the vertical and horizontal projections of $Z$ on $T_{x}N$
are given by%
\begin{equation*}
\left( d\pi \right) _{(x,u)}Z=Z^{r}\frac{\partial }{\partial x^{r}}%
|_{x},\quad K_{(x,u)}(Z)=\left( \overline{Z}^{r}+u^{s}Z^{t}\Gamma
_{st}^{r}\right) \frac{\partial }{\partial x^{r}}|_{x},
\end{equation*}%
where $\Gamma _{st}^{r}$ are components of the Levi-Civita connection on $N.$

On the other hand, to each vector field $X$ on $N$ there correspond uniquely
determined vector fields $X^{v}$ and $X^{h}$\ on $TN$ such that 
\begin{eqnarray*}
d\pi |_{(x,u)}(X^{v}) &=&0,\quad K|_{(x,u)}(X^{v})=X, \\
K|_{(x,u)}(X^{h}) &=&0,\quad d\pi |_{(x,u)}(X^{h})=X.
\end{eqnarray*}%
$X^{v}$ and $X^{h}$ are called the vertical lift and the horizontal lift of
a given $X$ to $TN,$ respectively.

In local coordinates $(\left( x^{r}\right) ,\left( u^{r}\right) ),$ $%
r=1,..,n,$ on $TN,$ the horizontal and vertical lifts of a vector field $%
X=X^{r}\frac{\partial }{\partial x^{r}}$ on $N$\ to $TN$ are vector fields
given respectively by%
\begin{equation}
X^{h}=X^{r}\frac{\partial }{\partial x^{r}}-u^{s}X^{t}\Gamma _{st}^{r}\frac{%
\partial }{\partial u^{r}},\quad X^{v}=X^{r}\frac{\partial }{\partial u^{r}}.
\label{Dec 4}
\end{equation}

Recall that for a given isometric immersion $f:M\longrightarrow N,$ we have
two tangent bundles $\pi _{N}:TN\longrightarrow N$ and $\pi
_{M}:TM\longrightarrow M,$ where the latter is the subbundle of the former.
Let $M,$ $N$ be two Riemannian manifolds with metrics $g_{M}$ and $g_{N}$
and Levi-Civita connections $\nabla _{M}$ and $\nabla _{N}$ respectively.
Then $T_{p}TM$ and $T_{p}TN$ have at a common point $p$ their own
decompositions into vertical and horizontal parts, i.e.%
\begin{equation*}
T_{p}TM=V_{p}TM\oplus H_{p}TM=V_{M}\oplus H_{M}
\end{equation*}%
and%
\begin{equation*}
T_{p}TN=V_{p}TN\oplus H_{p}TN=V_{N}\oplus H_{N},
\end{equation*}%
but neither $V_{M}\subset V_{N}$ nor $H_{M}\subset H_{N}$ need to hold along 
$TM.$ See also \cite{1}

Remark also that totally geodesic submanifolds of tangent bundle with $g$%
-natural metric\ are also studied in \cite{Abb Yamp} and \cite{Ewert 3}.

\subsection{Preliminaries on $g$-natural metrics}

In \cite{5} the class of $g$-natural metrics was defined. We have

\begin{lemma}
(\cite{5}, \cite{6}, \cite{7}) Let $(M,g)$ be a Riemannian manifold and $G$
be a $g$-natural metric on $TM.$ There exist functions $a_{j},$ $%
b_{j}:<0,\infty )\longrightarrow R,$ $j=1,2,3,$ such that for every $X,$ $Y,$
$u\in T_{x}M$%
\begin{multline*}
G_{(x,u)}(X^{h},Y^{h})=(a_{1}+a_{3})(r^{2})g_{x}(X,Y)+(b_{1}+b_{3})(r^{2})g_{x}(X,u)g_{x}(Y,u),
\\
G_{(x,u)}(X^{h},Y^{v})=G_{(x,u)}(X^{v},Y^{h})=a_{2}(r^{2})g_{x}(X,Y)+b_{2}(r^{2})g_{x}(X,u)g_{x}(Y,u),
\\
G_{(x,u)}(X^{v},Y^{v})=a_{1}(r^{2})g_{x}(X,Y)+b_{1}(r^{2})g_{x}(X,u)g_{x}(Y,u),
\end{multline*}%
where $r^{2}=g_{x}(u,u).$ For $\dim M=1$ the same holds for $b_{j}=0,$ $%
j=1,2,3.$
\end{lemma}

Setting $a_{1}=1,$ $a_{2}=a_{3}=b_{j}=0$ we obtain the Sasaki metric, while
setting $a_{1}=b_{1}=\frac{1}{1+r^{2}},$ $a_{2}=b_{2}=0=0,$ $a_{1}+a_{3}=1,$ 
$b_{1}+b_{3}=1$ we get the Cheeger-Gromoll one.

Following \cite{6} we put

\begin{enumerate}
\item $a(t)=a_{1}(t)\left( a_{1}(t)+a_{3}(t)\right) -a_{2}^{2}(t),$

\item $F_{j}(t)=a_{j}(t)+tb_{j}(t),$

\item $F(t)=F_{1}(t)\left[ F_{1}(t)+F_{3}(t)\right] -F_{2}^{2}(t)$

for all $t\in <0,\infty ).$
\end{enumerate}

We shall often abbreviate: $A=a_{1}+a_{3},$ $B=b_{1}+b_{3}.$

\begin{lemma}
\label{Lemma 9}(\cite{6}, Proposition 2.7) The necessary and sufficient
conditions for a $g$-natural metric $G$ on the tangent bundle of a
Riemannian manifold $(M,g)$ to be non-degenerate are 
\begin{equation*}
a(t)\neq 0,\quad F(t)\neq 0
\end{equation*}%
for all $t\in <0,\infty ).$ If $\dim M=1,$ this is equivalent to $a(t)\neq 0$
for all $t\in <0,\infty ).$

Moreover, $(TM,G)$ is Riemannian one if and only if%
\begin{equation*}
a(t)>0,\quad F(t)>0,\quad a_{1}(t)>0,\quad F_{1}(t)>0
\end{equation*}%
hold for all $t\in <0,\infty ).$
\end{lemma}

We also have

\begin{proposition}
\label{g-natural connection}(\cite{7-1}, \cite{7-2}) Let $(N,g)$ be a
Riemannian manifold, $\nabla $ its Levi-Civita connection and $R$ its
Riemann curvature tensor. If $G$ is a non-degenerate $g$-natural metric on $%
TN,$ then the Levi-Civita connection $\widetilde{\nabla }$ of $(TN,G)$ is
given at a point $(x,u)\in TN$ by%
\begin{equation*}
\left( \widetilde{\nabla }_{X^{h}}Y^{h}\right) _{(x,u)}=\left( \nabla
_{X}Y\right) _{(x,u)}^{h}+h\left\{ \mathbf{A}(u,X_{x},Y_{x})\right\}
+v\left\{ \mathbf{B}(u,X_{x},Y_{x})\right\} ,
\end{equation*}%
\begin{equation*}
\left( \widetilde{\nabla }_{X^{h}}Y^{v}\right) _{(x,u)}=\left( \nabla
_{X}Y\right) _{(x,u)}^{v}+h\left\{ \mathbf{C}(u,X_{x},Y_{x})\right\}
+v\left\{ \mathbf{D}(u,X_{x},Y_{x})\right\} ,
\end{equation*}%
\begin{equation*}
\left( \widetilde{\nabla }_{X^{v}}Y^{h}\right) _{(x,u)}=h\left\{ \mathbf{C}%
(u,Y_{x},X_{x})\right\} +v\left\{ \mathbf{D}(u,Y_{x},X_{x})\right\} ,
\end{equation*}%
\begin{equation*}
\left( \widetilde{\nabla }_{X^{v}}Y^{v}\right) _{(x,u)}=h\left\{ \mathbf{E}%
(u,X_{x},Y_{x})\right\} +v\left\{ \mathbf{F}(u,X_{x},Y_{x})\right\} ,
\end{equation*}%
where $\mathbf{A},$ $\mathbf{B},$ $\mathbf{C},$ $\mathbf{D},$ $\mathbf{E},$ $%
\mathbf{F}$ are some F-tensors defined on the product $TN\otimes TN\otimes
TN.$
\end{proposition}

\begin{remark}
Expressions for $\mathbf{A},$ $\mathbf{B},$ $\mathbf{C},$ $\mathbf{D},$ $%
\mathbf{E},$ $\mathbf{F}$ were presented for the first time in the original
papers (\cite{6}, \cite{7}). Unfortunately, they contain some misprints and
omissions. Therefore, for the correct form, we refer the reader to (\cite%
{7-1}, \cite{7-2}), see also (\cite{8}, \cite{11}).
\end{remark}

\subsection{Submanifolds}

Let $M$ be a manifold isometrically immersed in a pseudo-Riemannian manifold 
$N$ with metric $g.$ Denote by $\widetilde{\nabla }$ and $\nabla $ the
Levi-Civita connections of the metric $g$ on $N$ and that of the induced
metric on $M$ and by $D^{\perp }$ the connection induced in the normal
bundle $T^{\bot }M.$ Then the Gauss and Weingarten equations

\begin{equation*}
\widetilde{\nabla }_{X}Y=\nabla _{X}Y+H(X,Y),
\end{equation*}%
\begin{equation*}
\widetilde{\nabla }_{X}\eta =-A_{\eta }X+D_{X}^{\bot }Y,
\end{equation*}%
hold for all vectors fields $X,$ $Y$ tangent to $M$ and all vector fields $%
\eta $ normal to $M.$ Here $H(X,Y)$ is the second fundamental form which is
symmetric and takes values in $T^{\bot }M$\ while\ $A_{\eta }X$ is the shape
operator taking values in $TM.$ It is well known that $A_{\eta }$ and $H$
are related by%
\begin{equation*}
g(A_{\eta }X,Y)=g(\eta ,H(X,Y)).
\end{equation*}

$M$ is said to be totally geodesic if $H(X,Y)=0$ for all $X,$ $Y\in TM.$

For the local immersion $x^{r}=x^{r}(y^{a}),$ $r=1,...,n,$ $a=1,...,m,$ the
components of the Levi-Civita connection $\nabla $ of the induced metric $%
g_{ab}=g_{rs}B_{a}^{r}B_{b}^{s},$ $B_{a}^{r}=\frac{\partial x^{r}}{\partial
y^{a}},$ are%
\begin{equation*}
\Gamma _{ab}^{c}=\left[ B_{a.b}^{r}+\Gamma _{st}^{r}B_{a}^{s}B_{b}^{t}\right]
B_{r}^{c},\quad B_{r}^{c}=g^{cd}B_{d}^{t}g_{tr},
\end{equation*}%
where the dot denotes partial derivative with respect to $y^{b}.$

Similarly, the components of the connection $D^{\perp }$ are%
\begin{equation*}
\Gamma _{ay}^{x}=\left[ N_{y.a}^{r}+\Gamma _{st}^{r}B_{a}^{s}N_{y}^{t}\right]
N_{r}^{x},\quad N_{r}^{x}=g^{xy}N_{y}^{t}g_{tr},
\end{equation*}%
where $\eta _{z}=N_{z}^{r}\frac{\partial }{\partial x^{r}},$ $z=m+1,...,n$
are unit vector fields normal to $M.$

\subsubsection{Van der Waerden-Bortolotti covariant derivative}

Van der Waerden-Bortolotti covariant derivative $\overline{\nabla }$ is a
covariant differentiation of tensor fields of mixed types defined along a
submanifold $M$ isometrically immersed in a pseudo-Riemannian manifold $%
(N,g) $ and can be considered as a direct sum $\widetilde{\nabla }\oplus
\nabla \oplus \nabla ^{\bot }$ of the Levi-Civita connections of the metric $%
g$ on $N$, the one induced on $M$ and of the metric induced in normal bundle 
$T^{\bot }M.$

If $\widetilde{X},$ $X$ and $\eta $ are vector fields, respectively, tangent
to $N,$ tangent to $M$, normal to $M$ and $\widetilde{X}^{\ast },$ $X^{\ast
} $, $\eta ^{\ast }$ are respective 1-forms, then,%
\begin{multline*}
(\overline{\nabla }_{Y}T)\left( \widetilde{X},X,\eta ,\widetilde{X}^{\ast
},X^{\ast },\eta ^{\ast }\right) =Y\left( T\left( \widetilde{X},X,\eta ,%
\widetilde{X}^{\ast },X^{\ast },\eta ^{\ast }\right) \right) - \\
T\left( \widetilde{\nabla }_{Y}\widetilde{X},X,\eta ,\widetilde{X}^{\ast
},X^{\ast },\eta ^{\ast }\right) -T\left( \widetilde{X},X,\eta ,\widetilde{%
\nabla }_{Y}\widetilde{X}^{\ast },X^{\ast },\eta ^{\ast }\right) - \\
T\left( \widetilde{X},\nabla _{Y}X,\eta ,\widetilde{X}^{\ast },X^{\ast
},\eta ^{\ast }\right) -T\left( \widetilde{X},X,\eta ,\widetilde{X}^{\ast
},\nabla _{Y}X^{\ast },\eta ^{\ast }\right) - \\
T\left( \widetilde{X},X,\nabla _{Y}^{\bot }\eta ,\widetilde{X}^{\ast
},X^{\ast },\eta ^{\ast }\right) -T\left( \widetilde{X},X,\eta ,\widetilde{X}%
^{\ast },X^{\ast },\nabla _{Y}^{\bot }\eta ^{\ast }\right)
\end{multline*}%
for any vector field $Y$ tangent to $M$ and tensor field $T$ of mixed type $%
(3,3).$

Let $x^{k}=x^{k}(y^{a})$ be the local expression of the immersion, $%
B_{a}^{k}=\frac{\partial x^{k}}{\partial y^{a}}.$ Let $\eta _{x}=N_{x}^{r}%
\frac{\partial }{_{\partial x^{r}}},$ $x=m+1,...,n,$ be an orthonormal set
of vectors normal to $M.$

For the local coordinate vector fields $\frac{\partial }{\partial x^{k}}$
tangent to $N,$ $\frac{\partial }{\partial y^{a}}$ tangent to $M$, $\frac{%
\partial }{\partial v^{x}}$ normal to $M$ and the respective 1-forms $%
dx^{k}, $ $dy^{a},$ $dv^{x},$ denote by $\Gamma _{hk}^{l},$ $\Gamma
_{ab}^{c},$ $\Gamma _{ay}^{z}$ components of the connections $\widetilde{%
\nabla },$ $\nabla $ and $\nabla ^{\bot }.$ If 
\begin{equation*}
T=T_{hax}^{kby}\frac{\partial }{\partial x^{k}}\otimes dx^{h}\otimes \frac{%
\partial }{\partial y^{b}}\otimes dy^{a}\otimes \frac{\partial }{\partial
v^{y}}\otimes dv^{x},
\end{equation*}%
then%
\begin{multline*}
\nabla _{c}T_{hax}^{kby}=\partial _{b}T_{hax}^{kby}-\Gamma
_{hr}^{s}B_{c}^{r}T_{sax}^{kby}+\Gamma _{rs}^{k}B_{c}^{r}T_{hax}^{sby}- \\
\Gamma _{ca}^{d}T_{hdx}^{kby}+\Gamma _{cd}^{b}T_{hax}^{kdy}-\Gamma
_{cx}^{z}T_{haz}^{kby}+\Gamma _{cz}^{y}T_{hax}^{kbz},
\end{multline*}%
where $h,k,r,s=1,...,n,$ $a,b,c,d=1,...,m,$ $m<n,$ and $x,y,z=m+1,...,n.$

In particular, $\overline{\nabla }_{a}B_{b}^{r}$ and $\overline{\nabla }%
_{a}N_{x}^{r}$ give rise to the components of the second fundamental form
and the shape operator:%
\begin{equation*}
\overline{\nabla }_{b}B_{a}^{r}\partial _{r}=\left( B_{a.b}^{r}+\Gamma
_{st}^{r}B_{a}^{s}B_{b}^{t}-\Gamma _{ab}^{c}B_{c}^{r}\right) \partial
_{r}=h_{ab}^{z}N_{z}^{r}\partial _{r},
\end{equation*}%
\begin{equation}
\overline{\nabla }_{a}N_{x}^{r}\partial _{r}=\left( \partial
_{a}N_{x}^{r}+\Gamma _{st}^{r}B_{a}^{s}N_{x}^{t}-\Gamma
_{ax}^{y}N_{y}^{r}\right) \partial _{r}.  \label{WB 5}
\end{equation}%
In a free of coordinate notation we have respectively:%
\begin{equation*}
\overline{\nabla }_{X}Y=\widetilde{\nabla }_{X}Y-\nabla _{X}Y,
\end{equation*}

\begin{equation*}
\overline{\nabla }_{X}\eta =\widetilde{\nabla }_{X}\eta -D_{X}^{\bot }\eta .
\end{equation*}%
See also \cite{Yano Kon} and \cite{Yano}.

\subsubsection{Isometric immersion defined by normal bundle}

Let $f:M\longrightarrow N$ be an isometric immersion of a Riemannian
manifold $M$ into a Riemannian manifold $N.$ Suppose that the following
diagram commute 
\begin{equation*}
\begin{array}{ccccc}
(\pi _{N}^{-1}(U),(x^{r},u^{r}))~~TN & \boldsymbol{\longleftarrow -} & 
\widetilde{f} & \boldsymbol{--} & (\left( \pi _{M}^{\bot }\right)
^{-1}(V),(y^{a},v^{a}))~~T^{\bot }M \\ 
\begin{array}{c}
\boldsymbol{|} \\ 
\boldsymbol{|}%
\end{array}
&  &  &  & 
\begin{array}{c}
\boldsymbol{|} \\ 
\boldsymbol{|}%
\end{array}
\\ 
\pi _{N} &  &  &  & \pi _{M}^{\bot } \\ 
\begin{array}{c}
\boldsymbol{|} \\ 
\boldsymbol{\downarrow }%
\end{array}
&  &  &  & 
\begin{array}{c}
\boldsymbol{|} \\ 
\boldsymbol{\downarrow }%
\end{array}
\\ 
(U,(x^{r}))~~N & \boldsymbol{\longleftarrow -} & f & \boldsymbol{--} & 
(V,(y^{a}))~~M%
\end{array}%
,
\end{equation*}%
where $(U,(x^{r}))$ and $(V,(y^{a}))$ are coordinate neighbourhoods on $N$
and $M$ respectively, while the local expression for $f$ is:%
\begin{equation*}
f:x^{r}=x^{r}(y^{a}).
\end{equation*}

Let 
\begin{equation*}
\eta _{x}=N_{x}^{r}\frac{\partial }{\partial x^{r}},\qquad x=m+1,...,n
\end{equation*}%
be a set of orthonormal vectors normal to $M.$

The coordinate neighbourhoods on $TN$ and normal bundle $T^{\bot }M$ are
defined respectively by%
\begin{equation*}
\left( (x^{r}),(u^{r})\right) ,\qquad r=1,...,n,
\end{equation*}%
\begin{equation*}
\left( \left( y^{a}\right) ,(v^{x})\right) ,\qquad x=m+1,...,n,\quad
a=1,...,m,
\end{equation*}%
where $(u^{r})_{r=1,...,n},$ are components of the vector $u$ tangent to $N$
at a point with coordinates $(x^{r})$ and $(v^{x})_{x=m+1,...,n}$ are
components of the vector normal to $M$ at a point $x^{r}=x^{r}(y^{a}).$

If%
\begin{equation*}
f:M\longrightarrow N;\qquad \left( y^{a}\right) \mapsto x^{r}(y^{a})
\end{equation*}%
then

\begin{equation}
\widetilde{f}:x^{r}=x^{r}(y^{a}),\quad u^{r}=v^{x}N_{x}^{r}  \label{imersion}
\end{equation}%
defines locally an immersion into $TN.$

\subsubsection{Vectors tangent to $LM\label{vectors tangent to LM}$}

The coordinate vector fields tangent to $LM=\widetilde{f}(TM^{\perp })$ are%
\begin{equation*}
\frac{\partial }{\partial v^{x}}=N_{x}^{r}\frac{\partial }{\partial u^{r}}%
=\left( \eta _{x}\right) ^{v},
\end{equation*}

\begin{multline*}
\frac{\partial }{\partial y^{a}}=B_{a}^{r}\frac{\partial }{\partial x^{r}}%
+v^{z}\partial _{a}N_{z}^{r}\frac{\partial }{\partial u^{r}}\overset{(\ref%
{WB 5})}{=} \\
B_{a}^{r}\left( \frac{\partial }{\partial x^{r}}-v^{z}N_{z}^{t}\Gamma
_{tr}^{s}\frac{\partial }{\partial u^{s}}\right) +v^{z}\overline{\nabla }%
_{a}N_{z}^{r}\frac{\partial }{\partial u^{r}}+v^{z}\Gamma _{az}^{y}N_{y}^{r}%
\frac{\partial }{\partial u^{r}}\overset{(\ref{imersion}),(\ref{Dec 4})}{=}
\\
\left( \frac{\delta }{\delta y^{a}}\right) ^{h}+M_{a}^{r}\left( \frac{%
\partial }{\partial x^{r}}\right) ^{v}+N_{a}^{r}\left( \frac{\partial }{%
\partial x^{r}}\right) ^{v}= \\
\left( \frac{\delta }{\delta y^{a}}\right) ^{h}+\left( M_{a}\right)
^{v}+\left( N_{a}\right) ^{v},
\end{multline*}%
where $M_{a}=v^{z}\overline{\nabla }_{a}N_{z}^{r}\frac{\partial }{\partial
x^{r}}=v^{z}M_{az}$ are tangent to $M$ and $N_{a}=v^{z}\Gamma
_{az}^{y}N_{y}^{r}\frac{\partial }{\partial x^{r}}=N_{a}^{y}N_{y}^{r}\frac{%
\partial }{\partial x^{r}}$ are normal to $M.$

Along $M$ we also have%
\begin{equation*}
\nabla _{\delta _{a}}u=\nabla _{\delta _{a}}(v^{y}\eta _{y})=\nabla _{\delta
_{a}}\left( v^{y}N_{y}^{r}\frac{\partial }{\partial x^{r}}\right)
=M_{a}+N_{a}.
\end{equation*}

\section{Basic equations}

In this section we derive, using the equations of Gauss and Weingarten and
the formulas for the Levi-Civita connection on $(TN,G)$ with a
non-degenerate $g$-natural metric $G,$ the basic equations for the immersion
given by (\ref{imersion}) to be used throughout the paper. $H^{h}+V^{v}$ is
a unique decomposition of a vector field normal to $LM$ into its horizontal
and vertical parts, where $H$ and $V$ are vector fields along $M,$ not
necessary tangent to $M.$ $\nabla $ and $\widetilde{\nabla }$ denote the
Levi-Civita connections of the metric $g$ and $g$-natural non-degenerate
metric $G,$ respectively. $\widetilde{H}$ is the second fundamental form of
the immersion (\ref{imersion}). Finally, $R$ stands for the Riemann
curvature tensor of $g.$ The computations in this section were performed and
checked with Mathematica$^{\registered }$ software. In virtue of Proposition %
\ref{Prop 5}, the pairs of equations in each subsection must satisfy%
\begin{equation}
G(\widetilde{H}(\partial _{x},\partial _{y}),H^{h}+V^{v})-G(\widetilde{A}%
_{H^{h}+V^{v}}\partial _{x},\partial _{y})=0.  \label{second shape}
\end{equation}%
It also can be used to verify the correctness of computations.

\paragraph{Equation 1}

\begin{multline*}
G(\widetilde{\nabla }_{\partial _{x}}\partial _{y},H^{h}+V^{v})=G(\widetilde{%
H}(\partial _{x},\partial _{y}),H^{h}+V^{v})= \\
G(\widetilde{\nabla }_{\eta _{x}^{v}}\eta _{y}^{v},H^{h}+V^{v})= \\
G(h\{\mathbf{E(}u,\eta _{x},\eta _{y})\}+v\{\mathbf{F}(u,\eta _{x},\eta
_{y})\},H^{h}+V^{v})=
\end{multline*}%
\begin{multline}
b_{2}g(\eta _{x},\eta _{y})g(u,H)+(b_{1}-a_{1}^{\prime })g(\eta _{x},\eta
_{y})g(u,V)+  \label{Eq 52} \\
a_{1}^{\prime }g(\eta _{x},V)g(u,\eta _{y})+a_{1}^{\prime }g(\eta
_{y},V)g(u,\eta _{x})+ \\
\left( a_{2}^{\prime }+\frac{b_{2}}{2}\right) \left( g(\eta _{x},H)g(u,\eta
_{y})+g(\eta _{y},H)g(u,\eta _{x})\right) + \\
g(u,\eta _{x})g(u,\eta _{y})g(u,b_{1}^{\prime }V+2b_{2}^{\prime }H).
\end{multline}%
\begin{multline*}
G(\widetilde{\nabla }_{\partial _{x}}(H^{h}+V^{v}),\partial _{y})=G(-%
\widetilde{A}_{H^{h}+V^{v}}\partial _{x},\partial _{y})= \\
G(h\{\mathbf{C}(u,H,\eta _{x})\}+v\{\mathbf{D}(u,H,\eta _{x})\}+h\{\mathbf{E}%
(u,\eta _{x},V)\}+v\{\mathbf{F}(u,\eta _{x},V)\},\eta _{y}^{v})=
\end{multline*}%
\begin{multline*}
(b_{1}-a_{1}^{\prime })g(\eta _{x},V)g(u,\eta _{y})+a_{1}^{\prime }g(\eta
_{y},V)g(u,\eta _{x})+a_{1}^{\prime }g(\eta _{x},\eta _{y})g(u,V)+ \\
\left( a_{2}^{\prime }-\frac{b_{2}}{2}\right) \left( g(\eta _{y},H)g(u,\eta
_{x})-g(\eta _{x},H)g(u,\eta _{y})\right) + \\
b_{1}^{\prime }g(u,\eta _{x})g(u,\eta _{y})g(u,V).
\end{multline*}%
In virtue of the equality (\ref{second shape}) the above two equations yield%
\begin{equation*}
g(u,\eta _{x})g(\eta _{y},T)-g(u,\eta _{y})g(\eta _{x},T)=0,
\end{equation*}%
where $T=(b_{1}^{\prime }-2a_{1}^{\prime })V+(b_{2}^{\prime }-2a_{2}^{\prime
})H.$

\paragraph{Equation 2}

\begin{multline*}
G(\widetilde{\nabla }_{\partial _{x}}\partial _{a},H^{h}+V^{v})=G(\widetilde{%
H}(\partial _{x},\partial _{a}),H^{h}+V^{v})= \\
G(\widetilde{\nabla }_{\eta _{x}^{v}}(\delta
_{a}^{h}+M_{a}^{v}+N_{a}^{v}),H^{h}+V^{v})= \\
G(h\{\mathbf{C}(u,\delta _{a},\eta _{x})\}+v\{\mathbf{D}(u,\delta _{a},\eta
_{x})\},H^{h}+V^{v})+ \\
G(h\{\mathbf{E(}u,\eta _{x},M_{a}+N_{a})\}+v\{\mathbf{F}(u,\eta
_{x},M_{a}+N_{a})\},H^{h}+V^{v})=
\end{multline*}%
\begin{multline}
-\frac{1}{2}a_{1}R(H,\delta _{a},u,\eta _{x})+A^{\prime }g(H,\delta
_{a})g(u,\eta _{x})+\left( a_{2}^{\prime }-\frac{b_{2}}{2}\right) g(V,\delta
_{a})g(u,\eta _{x})+  \label{Eq 62} \\
(b_{1}-a_{1}^{\prime })g(u,V)g(N_{a},\eta _{x})+b_{2}g(u,H)g(N_{a},\eta
_{x})+ \\
a_{1}^{\prime }g(u,\eta _{x})g(V,M_{a}+N_{a})+a_{1}^{\prime
}g(u,N_{a})g(V,\eta _{x})+ \\
\left( a_{2}^{\prime }+\frac{b_{2}}{2}\right) \left[ g(H,M_{a}+N_{a})g(u,%
\eta _{x})+g(H,\eta _{x})g(u,N_{a})\right] + \\
g(u,b_{1}^{\prime }V+2b_{2}^{\prime }H)g(u,\eta _{x})g(u,N_{a}).
\end{multline}

\begin{multline*}
G(\widetilde{\nabla }_{\partial _{x}}(H^{h}+V^{v}),\partial _{a})=G(-%
\widetilde{A}_{H^{h}+V^{v}}\partial _{x},\partial _{a})= \\
G(\widetilde{\nabla }_{\partial _{x}}(H^{h}+V^{v}),\delta
_{a}^{h}+M_{a}^{v}+N_{a}^{v})= \\
G(h\{\mathbf{C}(u,H,\eta _{x})\}+v\{\mathbf{D}(u,H,\eta _{x})\},\delta
_{a}^{h}+M_{a}^{v}+N_{a}^{v})+ \\
G(h\{\mathbf{E}(u,\eta _{x},V)\}+v\{\mathbf{F}(u,\eta _{x},V)\},\delta
_{a}^{h}+M_{a}^{v}+N_{a}^{v})=
\end{multline*}%
\begin{multline}
\frac{1}{2}a_{1}R(H,\delta _{a},u,\eta _{x})+(b_{1}-a_{1}^{\prime
})g(u,N_{a})g(V,\eta _{x})+a_{1}^{\prime }g(u,\eta _{x})g(V,M_{a}+N_{a})+
\label{Eq 58} \\
a_{1}^{\prime }g(u,V)g(N_{a},\eta _{x})+\left( a_{2}^{\prime }+\frac{b_{2}}{2%
}\right) g(u,\eta _{x})g(V,\delta _{a})+ \\
\left( a_{2}^{\prime }-\frac{b_{2}}{2}\right) \left[ g(H,M_{a}+N_{a})g(u,%
\eta _{x})-g(H,\eta _{x})g(u,N_{a})\right] + \\
A^{\prime }g(H,\delta _{a})g(u,\eta _{x})+b_{1}^{\prime }g(u,\eta
_{x})g(u,N_{a})g(u,V).
\end{multline}

Hence, in virtue of (\ref{second shape}), we get%
\begin{multline*}
g(u,b_{1}V+b_{2}H)g(N_{a},\eta _{x})+2g(u,\eta _{x})g(A^{\prime
}H+a_{2}^{\prime }V,\delta _{a})+ \\
2g(u,\eta _{x})g(a_{1}^{\prime }V+a_{2}^{\prime }H,M_{a}+N_{a})+ \\
g(u,N_{a})\left[ g(b_{1}V+b_{2}H,\eta _{x})+2g(u,\eta _{x})g(b_{1}^{\prime
}V+b_{2}^{\prime }H,u)\right] =0.
\end{multline*}

\paragraph{Equation 3}

\begin{multline*}
G(\widetilde{\nabla }_{\partial _{a}}\partial _{x},H^{h}+V^{v})=G(\widetilde{%
H}(\partial _{a},\partial _{x}),H^{h}+V^{v})= \\
G(\widetilde{\nabla }_{(\delta _{a}^{h}+M_{a}^{v}+N_{a}^{v})}(\eta
_{x})^{v},H^{h}+V^{v})=
\end{multline*}%
\begin{multline*}
G(\left( \nabla _{\delta _{a}}\eta _{x}\right) ^{v}+h\{\mathbf{C}(u,\delta
_{a},\eta _{x})\}+v\{\mathbf{D}(u,\delta _{a},\eta _{x})\},H^{h}+V^{v})+ \\
G(h\{\mathbf{E(}u,\eta _{x},M_{a}+N_{a})\}+v\{\mathbf{F}(u,\eta
_{x},M_{a}+N_{a})\},H^{h}+V^{v}).
\end{multline*}%
Since $\widetilde{H}(\partial _{a},\partial _{x})$ is symmetric, comparing
the last equation with (\ref{Eq 62}), we obtain%
\begin{equation}
G\left( \left( \nabla _{\delta _{a}}\eta _{x}\right) ^{v},H^{h}+V^{v}\right)
=0.  \label{Eq 18}
\end{equation}

\begin{multline*}
G(\widetilde{\nabla }_{\partial _{a}}(H^{h}+V^{v}),\partial _{x})=G(-%
\widetilde{A}_{H^{h}+V^{v}}\partial _{a},\partial _{x})= \\
G(\widetilde{\nabla }_{(\delta
_{a}^{h}+M_{a}^{v}+N_{a}^{v})}(H^{h}+V^{v}),\eta _{x}^{v})= \\
G(\left( \nabla _{\delta _{a}}H\right) ^{h}+h\{\mathbf{A}(u,\delta
_{a},H)\}+v\{\mathbf{B}(u,\delta _{a},H)\},\eta _{x}^{v})+ \\
G(\left( \nabla _{\delta _{a}}V\right) ^{v}+h\{\mathbf{C}(u,\delta
_{a},V)\}+v\{\mathbf{D}(u,\delta _{a},V)\},\eta _{x}^{v})+ \\
G(h\{\mathbf{C}(u,H,M_{a}+N_{a})\}+v\{\mathbf{D}(u,H,M_{a}+N_{a})\},\eta
_{x}^{v})+ \\
G(h\{\mathbf{E}(u,M_{a}+N_{a},V)\}+v\{\mathbf{F}(u,M_{a}+N_{a},V)\},\eta
_{x}^{v})=
\end{multline*}%
\begin{multline}
\frac{1}{2}a_{1}R(H,\delta _{a},u,\eta _{x})+a_{1}g(\eta _{x},\nabla
_{\delta _{a}}V)+b_{1}g(u,\eta _{x})g(u,\nabla _{\delta _{a}}V)+
\label{Eq 69} \\
a_{2}g(\eta _{x},\nabla _{\delta _{a}}H)+b_{2}g(u,\eta _{x})g(u,\nabla
_{\delta _{a}}H)-A^{\prime }g(u,\eta _{x})g(H,\delta _{a})+ \\
(b_{1}-a_{1}^{\prime })g(u,\eta _{x})g(V,M_{a}+N_{a})+a_{1}^{\prime
}g(u,N_{a})g(V,\eta _{x})+a_{1}^{\prime }g(u,V)g(N_{a},\eta _{x})+ \\
\left( a_{2}^{\prime }-\frac{b_{2}}{2}\right) \left[ g(H,\eta
_{x})g(u,N_{a})-g(H,M_{a}+N_{a})g(u,\eta _{x})-g(u,\eta _{x})g(V,\delta _{a})%
\right] + \\
b_{1}^{\prime }g(u,V)g(u,N_{a})g(u,\eta _{x}).
\end{multline}

\paragraph{Equation 4}

\QTP{Body Math}
\begin{multline*}
G(\widetilde{\nabla }_{\partial _{a}}\partial _{b},H^{h}+V^{v})=G(\widetilde{%
H}(\partial _{a},\partial _{b}),H^{h}+V^{v})= \\
G(\widetilde{\nabla }_{(\delta _{a}^{h}+M_{a}^{v}+N_{a}^{v})}(\delta
_{b}^{h}+M_{b}^{v}+N_{b}^{v}),H^{h}+V^{v})= \\
G(\left( \nabla _{\delta _{a}}\delta _{b}\right) ^{h}+h\left\{ \mathbf{A}%
(u,\delta _{a},\delta _{b})\right\} +v\left\{ \mathbf{B}(u,\delta
_{a},\delta _{b})\right\} + \\
\left( \nabla _{\delta _{a}}(M_{b}+N_{b})\right) ^{v}+h\left\{ \mathbf{C}%
(u,\delta _{a},M_{b}+N_{b})\right\} +v\left\{ \mathbf{D}(u,\delta
_{a},M_{b}+N_{b})\right\} + \\
h\left\{ \mathbf{C}(u,\delta _{b},M_{a}+N_{a})\right\} +v\left\{ \mathbf{D}%
(u,\delta _{b},M_{a}+N_{a})\right\} + \\
h\left\{ \mathbf{E}(u,M_{a}+N_{a},M_{b}+N_{b})\right\} +v\left\{ \mathbf{F}%
(u,M_{a}+N_{a},M_{b}+N_{b})\right\} ,H^{h}+V^{v})=
\end{multline*}%
\begin{multline}
-a_{2}R(H,\delta _{b},u,\delta _{a})- \\
\frac{1}{2}a_{1}R(H,\delta _{a},u,M_{b}+N_{b})-\frac{1}{2}a_{1}R(H,\delta
_{b},u,M_{a}+N_{a})-\frac{1}{2}a_{1}R(u,V,\delta _{a},\delta _{b})+ \\
g(AH+a_{2}V,\nabla _{\delta _{a}}\delta _{b})+g(BH+b_{2}V,u)g(u,\nabla
_{\delta _{a}}\delta _{b})+ \\
\frac{1}{2}Bg(H,u)\left( g(M_{a},\delta _{b})+g(M_{b},\delta _{a})\right) +
\\
g(a_{1}V+a_{2}H,\nabla _{\delta
_{a}}(M_{b}+N_{b}))+g(b_{2}H+b_{1}V,u)g(u,\nabla _{\delta
_{a}}(M_{b}+N_{b}))+ \\
+g(b_{2}H+(b_{1}-a_{1}^{\prime })V,u)(g(M_{a},M_{b})+g(N_{a},N_{b}))+ \\
A^{\prime }\left[ g(H,\delta _{a})g(u,N_{b})+g(H,\delta
_{b})g(u,N_{a})-g(V,u)g(\delta _{a},\delta _{b})\right] + \\
a_{1}^{\prime }g(u,N_{a})g(V,M_{b}+N_{b})+a_{1}^{\prime
}g(u,N_{b})g(V,M_{a}+N_{a})+ \\
\left( a_{2}^{\prime }-\frac{b_{2}}{2}\right) \left\{ g(u,N_{b})g(V,\delta
_{a})+g(u,N_{a})g(V,\delta _{b})-g(u,V)\left[ g(M_{b},\delta
_{a})+g(M_{a},\delta _{b})\right] \right\} + \\
\left( a_{2}^{\prime }+\frac{b_{2}}{2}\right) \left(
g(H,M_{a}+N_{a})g(u,N_{b})+g(H,M_{b}+N_{b})g(u,N_{a})\right) + \\
g(u,b_{1}^{\prime }V+2b_{2}^{\prime }H)g(u,N_{a})g(u,N_{b}).
\label{Eq 1 Eq 57}
\end{multline}

\QTP{Body Math}
\begin{multline*}
G(\widetilde{\nabla }_{\partial _{a}}(H^{h}+V^{v}),\partial _{b})=G(-%
\widetilde{A}_{H^{h}+V^{v}}\partial _{a},\partial _{b})= \\
G(\widetilde{\nabla }_{(\delta
_{a}^{h}+M_{a}^{v}+N_{a}^{v})}(H^{h}+V^{v}),\eta _{b}^{v})=
\end{multline*}%
\begin{multline*}
a_{2}R(H,\delta _{b},u,\delta _{a})+ \\
\frac{1}{2}a_{1}R(H,\delta _{a},u,M_{b}+N_{b})+\frac{1}{2}a_{1}R(H,\delta
_{b},u,M_{a}+N_{a})+\frac{1}{2}a_{1}R(u,V,\delta _{a},\delta _{b})+ \\
Ag(\delta _{b},\nabla _{\delta _{a}}H)+a_{2}g(\delta _{b},\nabla _{\delta
_{a}}V)+a_{2}g(M_{b}+N_{b},\nabla _{\delta _{a}}H)+a_{1}g(M_{b}+N_{b},\nabla
_{\delta _{a}}V)+ \\
\frac{1}{2}Bg(H,u)\left( g(M_{a},\delta _{b})-g(M_{b},\delta _{a})\right) +
\\
g(u,N_{b})(b_{2}g(u,\nabla _{\delta _{a}}H)+b_{1}g(u,\nabla _{\delta
_{a}}V))+ \\
A^{\prime }(-g(H,\delta _{a})g(u,N_{b})+g(H,\delta
_{b})g(u,N_{a})+g(V,u)g(\delta _{a},\delta _{b}))+ \\
(b_{1}-a_{1}^{\prime })g(u,N_{b})g(V,M_{a}+N_{a})+a_{1}^{\prime
}g(u,N_{a})g(V,M_{b}+N_{b})+
\end{multline*}%
\begin{multline*}
a_{1}^{\prime }g(V,u)\left[ g(M_{a},M_{b})+g(N_{a},N_{b})\right] + \\
\left( a_{2}^{\prime }-\frac{b_{2}}{2}\right) (g(u,V)g(M_{b},\delta
_{a})-g(u,N_{b})g(V,\delta _{a}))+ \\
\left( a_{2}^{\prime }+\frac{b_{2}}{2}\right) (g(u,V)g(M_{a},\delta
_{b})+g(u,N_{a})g(V,\delta _{b}))+ \\
\left( a_{2}^{\prime }-\frac{b_{2}}{2}\right) \left[
g(-H,M_{a}+N_{a})g(u,N_{b})+g(H,M_{b}+N_{b})g(u,N_{a})\right] + \\
b_{1}^{\prime }g(u,V)g(u,N_{a})g(u,N_{b}).
\end{multline*}

Applying (\ref{second shape}), we find%
\begin{multline*}
0=g(AH+a_{2}V,\nabla _{\delta _{a}}\delta _{b})+ \\
g(BH+b_{2}V,u)g(u,\nabla _{\delta _{a}}\delta _{b})+g(a_{2}H+a_{1}V,\nabla
_{\delta _{a}}(M_{b}+N_{b}))+ \\
a_{2}g(M_{b}+N_{b},\nabla _{\delta _{a}}H)+a_{1}g(M_{b}+N_{b},\nabla
_{\delta _{a}}V)+ \\
Ag(\delta _{b},\nabla _{\delta _{a}}H)+a_{2}g(\delta _{b},\nabla _{\delta
_{a}}V)+g(BH+b_{2}V,u)g(M_{a},\delta _{b})+ \\
g(u,b_{2}H+b_{1}V)g(u,\nabla _{\delta _{a}}(M_{b}+N_{b})+ \\
g(u,N_{b})\left[ b_{1}g(u,\nabla _{\delta _{a}}V)+b_{2}g(u,\nabla _{\delta
_{a}}H)+g(b_{2}H+b_{1}V,M_{a}+N_{a})\right] + \\
g(u,b_{2}H+b_{1}V)\left[ g(M_{a},M_{b})+g(N_{a},N_{b})\right]
+2g(u,N_{a})\times \\
\left[ g(A^{\prime }H+a_{2}^{\prime }V,\delta _{b})+g(a_{2}^{\prime
}H+a_{1}^{\prime }V,M_{b}+N_{b})+g(u,N_{b})g(b_{2}^{\prime }H+b_{1}^{\prime
}V,u)\right] .
\end{multline*}

\section{Main results}

The first proposition of this section establishes a number of various
relations that allow us to show that the right hand sides of the pairs of
equations in each subsection of the former section satisfy (\ref{second
shape}). The results are presented in Proposition \ref{Prop 5}. Theorem \ref%
{Th 6} states the condition sufficient for the space normal to $LM$ being
spanned by lifts of vectors tangent to $M.$ The main results are presented
in Theorems \ref{th 10} and \ref{th 14}.

\begin{proposition}
\label{Prop 5}Let $\widetilde{f}$ be the immersion given by (\ref{imersion})
defined by the isometric immersion $f:M\rightarrow (N,g)$ into a Riemannian
manifold. Suppose, moreover, that $TN$ is endowed with non-degenerate $g$%
-natural metric $G.$ Then in the notation as above the following identities
are satisfied.

\begin{enumerate}
\item 
\begin{equation}
g(\eta _{x},S)=0,  \label{20-1}
\end{equation}%
where $S=a_{2}H+a_{1}V+g(u,b_{2}H+b_{1}V)u.$

\item 
\begin{equation}
g(u,S)=g(N_{a},S)=0.  \label{20-2}
\end{equation}

\item 
\begin{equation}
g(\nabla _{\delta _{a}}\eta _{x},S)=0.  \label{20-3}
\end{equation}

\item 
\begin{equation}
g(\eta _{x},\nabla _{\delta _{a}}S)=0.  \label{20-4}
\end{equation}

\item 
\begin{equation}
g(u,\nabla _{\delta _{a}}S)=g(N_{b},\nabla _{\delta _{a}}S)=0.  \label{20-5}
\end{equation}

\item 
\begin{equation}
g(\nabla _{\delta _{a}}u,S)=g(\nabla _{\delta _{a}}N_{b},S)=0.  \label{20-6}
\end{equation}

\item 
\begin{equation}
g(\delta _{a},AH+a_{2}V)=g(M_{a},a_{2}H+a_{1}V)=0.  \label{20-7}
\end{equation}

\item 
\begin{equation}
g(\delta _{a},A^{\prime }H+a_{2}^{\prime }V)=g(M_{a},a_{2}^{\prime
}H+a_{1}^{\prime }V)=0.  \label{20-8}
\end{equation}

\item 
\begin{equation}
g(M_{b},M_{a})+g(u,\nabla _{\delta _{a}}M_{b})=0.  \label{20-9}
\end{equation}

\item 
\begin{equation}
g(M_{b},\delta _{a})+g(u,\nabla _{\delta _{b}}\delta _{a})=0.  \label{20-10}
\end{equation}

\item 
\begin{equation}
g(\nabla _{\delta _{a}}\eta _{x},\delta _{b})+g(\eta _{x},\nabla _{\delta
_{a}}\delta _{b})=0.  \label{20-11}
\end{equation}

Moreover, if $M$ is not a hypersurface of $N,$ then

\item 
\begin{equation}
X_{u}=g(u,b_{2}H+b_{1}V)=0.  \label{20-12}
\end{equation}

\item 
\begin{equation}
X_{\eta _{x}}=g(\eta _{x},b_{2}H+b_{1}V)=0.  \label{20-16}
\end{equation}

\item 
\begin{equation}
Y_{\eta _{x}}=g(\eta _{x},b_{2}^{\prime }H+b_{1}^{\prime }V)=0.
\label{20-17}
\end{equation}

\item 
\begin{equation}
Y_{u}=g(u,b_{2}^{\prime }H+b_{1}^{\prime }V)=0.  \label{20-15}
\end{equation}

\item 
\begin{equation}
Z_{\eta _{x}}=g(\eta _{x},a_{2}^{\prime }H+a_{1}^{\prime }V)=0,\quad
Z_{u}=g(u,a_{2}^{\prime }H+a_{1}^{\prime }V)=0.  \label{20-14}
\end{equation}

\item 
\begin{equation}
g(\eta _{x},a_{2}H+a_{1}V)=0.  \label{20-13}
\end{equation}

Finaly

\item 
\begin{equation*}
S=a_{2}H+a_{1}V.
\end{equation*}
\end{enumerate}
\end{proposition}

\begin{proof}
(\ref{20-1}) results from

\begin{equation*}
G(\partial _{x},H^{h}+V^{v})=G(\eta _{x}^{v},H^{h}+V^{v})=0.
\end{equation*}%
Then (\ref{20-2}) is obvious since $u=v^{y}N_{y}^{r}\partial _{r}=v^{y}\eta
_{y}$ and $N_{a}=N_{a}^{y}\eta _{y}$ are normal to $M.$ Now (\ref{20-3}) is
a consequence of (\ref{Eq 18}), whence, by (\ref{20-1}), (\ref{20-4})
results.

Once again, by orthogonality of $u$ and $N_{a}$ with respect to $M,$ we have
(\ref{20-5}). Consequently, in virtue\ of (\ref{20-2}), we obtain (\ref{20-6}%
).

Observe that the identity 
\begin{equation}
(\ref{Eq 58})-(\ref{Eq 69})+g(\eta _{x},\nabla _{\delta _{a}}S)-\sum_{y}%
\frac{\partial }{\partial v^{y}}g(\eta _{x},S)N_{a}^{y}=0  \label{20-7a}
\end{equation}%
gives%
\begin{equation*}
g(\delta _{a},A^{\prime }H+a_{2}^{\prime }V)+g(M_{a},a_{2}^{\prime
}H+a_{1}^{\prime }V)=0.
\end{equation*}%
On the other hand, relations%
\begin{equation*}
G(\partial _{a},H^{h}+V^{v})=G(\delta
_{a}^{h}+M_{a}^{v}+N_{a}^{v},H^{h}+V^{v})=0
\end{equation*}%
and (\ref{20-2}) yield%
\begin{equation}
g(\delta _{a},AH+a_{2}V)+g(M_{a},a_{2}H+a_{1}V)=0.  \label{20-7b}
\end{equation}%
Differentiating (\ref{20-7b}) with respect to $v^{x}\ $and using (\ref{20-7a}%
) we find%
\begin{equation*}
g(M_{ax},a_{2}H+a_{1}V)=0,
\end{equation*}%
where $M_{ax}=\nabla _{\delta _{a}}N_{x}^{r}\partial _{r}.$ Consequently, (%
\ref{20-7b}) yields (\ref{20-7}). Hence, by differentiating with respect to $%
v^{x},$ (\ref{20-8}) results.

Since $M_{a},$ $\delta _{a}$ are tangent to $M$ and $u,$ $\eta _{x}$ are
normal, by covariant differentiation of $g(u,M_{a})=0,$ $g(u,\delta _{a})=0,$
$g(\eta _{x},\delta _{a})=0$ we get (\ref{20-9}) - (\ref{20-11}).

Differentiating (\ref{20-1}) with respect to $v^{y}$ we get%
\begin{equation*}
g(\eta _{x},\eta _{y})X_{u}+2g(\eta _{x},u)g(u,\eta _{y})Y_{u}+g(\eta
_{x},u)X_{\eta _{y}}+2g(u,\eta _{y})Z_{\eta _{x}}=0,
\end{equation*}%
where $X_{u}=g(u,b_{2}H+b_{1}V),$ $X_{\eta _{x}}=g(\eta _{x},b_{2}H+b_{1}V),$
$Y_{u}=g(u,b_{2}^{\prime }H+b_{1}^{\prime }V),$ $Z_{u}=g(u,a_{2}^{\prime
}H+a_{1}^{\prime }V)$ and $Z_{\eta _{x}}=g(\eta _{x},a_{2}^{\prime
}H+a_{1}^{\prime }V).$

Transvecting in turn with $v^{x},$ $v^{y},$ $v^{x}v^{y}$ and, finally,
contracting wit $g^{xy}$ we get for each $x=m+1,...,n$ a system of four
equations:%
\begin{eqnarray}
v_{x}X_{u}+2r^{2}v_{x}Y_{u}+r^{2}X_{\eta _{x}}+2v_{x}Z_{u} &=&0,  \notag \\
v_{x}X_{u}+2r^{2}v_{x}Y_{u}+v_{x}X_{u}+2r^{2}Z_{\eta _{x}} &=&0,  \notag \\
r^{2}(X_{u}+r^{2}Y_{u}+Z_{u}) &=&0,  \notag \\
(n-m+1)X_{u}+2r^{2}Y_{u}+2Z_{u} &=&0,  \label{20-47}
\end{eqnarray}%
where $v_{x}=g(u,\eta _{x}).$ Solving it with respect to $X,$ $X_{\eta
_{x}}, $ $Y,$ $Z_{\eta _{x}}$ we obtain 
\begin{equation}
X_{u}=X_{\eta _{x}}=0,\quad Y_{u}=-\frac{Z_{u}}{r^{2}},\quad Z_{\eta _{x}}=-%
\frac{v_{x}Z_{u}}{r^{2}}  \label{20-48}
\end{equation}%
for any $u=v^{x}\eta _{x}\neq 0.$ By continuity, $X_{u}=X_{\eta _{x}}=0$
hold for any $u.$ Then $\frac{\partial }{\partial v^{z}}X_{\eta
_{x}}=Y_{\eta _{x}}g(\eta _{z},u)=0$ for all $u\neq 0,$ whence, in virtue of
continuity, $Y_{\eta _{x}}=0$ for any $u.$ Consequently, we have $Y_{u}=0.$
Now, $Z_{u}=0$ follows from (\ref{20-47}) and $Z_{\eta _{x}}=0$ results from
(\ref{20-48}). Finally, (\ref{20-13}) is a consequence of (\ref{20-1}) and (%
\ref{20-12}). Thus the lemma is proved.
\end{proof}

\begin{theorem}
\label{Th 6}Let $(x,u)$ be a point of $LM$ immersed in $TN.$ If $\func{codim}%
M>1$ and%
\begin{equation}
a_{2}b_{1}-a_{1}b_{2}\neq 0  \label{inequality 1}
\end{equation}%
at $t=g(u,u),$ then the normal space at $(x,u)\in LM$ is spanned by lifts of
vectors tangent to $M.$
\end{theorem}

\begin{proof}
It results from the identities (\ref{20-13}) and (\ref{20-16}). Note that
other conditions, similar to that of (\ref{inequality 1}) can be deduced in
the same way from (\ref{20-13}), (\ref{20-14}), (\ref{20-16}) and (\ref%
{20-17}).
\end{proof}

We shall prove that the conditions $\func{codim}M>1$ and (\ref{inequality 1}%
) are essential in that sense that there exist immersion $f:M\longrightarrow
N,$ $M$ being a hypersurface of $N,$ and a metric $G$ on $TN$ satisfying $%
a_{2}b_{1}-a_{1}b_{2}=0$ such that the normal component of at least one of
the vectors $H,$ $V$ does not vanish.

\begin{example}
Let $f:S^{1}\longrightarrow \left( R^{2},Euclid\ metric\right) $ be the
immersion given by $f(t)=[\cos t,\sin t].$ The vector tangent to $S^{1}$ is $%
\mathbf{s}=[-\sin t,\cos t]$ and the normal one is $\mathbf{n}=[\cos t,\sin
t].$ Then%
\begin{multline*}
\mathbf{s}^{v}=[0,0,-\sin t,\cos t],\qquad \mathbf{s}^{h}=[-\sin t,\cos
t,0,0], \\
\mathbf{n}^{v}=[0,0,\cos t,\sin t],\qquad \mathbf{n}^{h}=[\cos t,\sin t,0,0].
\end{multline*}%
The vectors tangent to $L\left( S^{1}\right) $ are%
\begin{equation*}
\frac{\partial }{\partial t}=\mathbf{s}^{h}+v\mathbf{s}^{v},\quad \frac{%
\partial }{\partial v}=\mathbf{n}^{v}.
\end{equation*}%
Suppose%
\begin{equation*}
H^{h}+V^{v}=\alpha \mathbf{s}^{h}+\beta \mathbf{n}^{h}+\gamma \mathbf{s}%
^{v}+\delta \mathbf{n}^{v}
\end{equation*}%
and consider a non-degenerate $g$-natural metric on $TR^{2}\ $such that $%
B=b_{1}=b_{2}=0.$

Then%
\begin{equation*}
G(\frac{\partial }{\partial v},H^{h}+V^{v})=G(\mathbf{n}%
^{v},H^{h}+V^{v})=a_{2}g(\mathbf{n},H)+a_{1}g(\mathbf{n},V)=a_{2}\beta
+a_{1}\delta
\end{equation*}%
and%
\begin{equation*}
G(\frac{\partial }{\partial t},H^{h}+V^{v})=G(\mathbf{s}^{h},H^{h}+V^{v})=Ag(%
\mathbf{s},H)+a_{2}g(\mathbf{s},V)=A\alpha +a_{2}\gamma .
\end{equation*}%
We put%
\begin{equation*}
\alpha =-\frac{a_{2}+va_{1}}{A+va_{2}}\gamma \neq 0,\qquad \beta =-\frac{%
a_{1}}{a_{2}}\delta \neq 0.
\end{equation*}
\end{example}

The restriction on codimension can be omitted as the next proposition shows.

\begin{proposition}
If $a_{2}=b_{2}=0$ for all $t\in <0,\infty ),$ then the normal bundle of $LM$
at a point $(x,u),$ $x\in M,$ $u\in T_{x}M,$ is spanned by the vectors $\eta
_{y}^{h}$ and $\delta _{a}^{v}-\frac{a_{1}}{A}(\nabla _{\delta _{a}}u)^{h}.$
\end{proposition}

\begin{proof}
Direct calculation. Property (\ref{20-10}) is applied.
\end{proof}

Applying Proposition \ref{Prop 5} to (\ref{Eq 52}), (\ref{Eq 62}) and (\ref%
{Eq 1 Eq 57}) we obtain

\begin{theorem}
\label{th 10}Let $M,$ $\func{codim}M>1,$ be a submanifold isometrically
immersed in a manifold $N$. Then along $LM$ we have:

\begin{enumerate}
\item 
\begin{multline*}
\widetilde{G}\left[ H^{h}+V^{v},\widetilde{\nabla }_{\partial _{x}}\partial
_{y}\right] =\widetilde{G}\left[ H^{h}+V^{v},\widetilde{H}(\partial
_{x},\partial _{y})\right] = \\
g\left[ H,\frac{b_{2}}{2}g\left( u,\eta _{x}\right) \eta _{y}+\frac{b_{2}}{2}%
g\left( u,\eta _{y}\right) \eta _{x}+a_{2}^{\prime }g\left( \eta _{x},\eta
_{y}\right) u+b_{2}^{\prime }g\left( u,\eta _{x}\right) g\left( u,\eta
_{y}\right) u\right] ,
\end{multline*}

\item 
\begin{multline*}
\widetilde{G}\left[ H^{h}+V^{v},\widetilde{\nabla }_{\partial _{x}}\partial
_{b}\right] =\widetilde{G}\left[ H^{h}+V^{v},\widetilde{H}(\partial
_{x},\partial _{a})\right] = \\
g\left[ H,\frac{1}{2}a_{1}R(u,\eta _{x})\delta _{a}\right] + \\
g\left[ H,\frac{b_{2}}{2}g(u,N_{a})\eta _{x}+\frac{b_{2}}{2}g(u,\eta
_{x})(M_{a}+N_{a})+a_{2}^{\prime }g(N_{a},\eta _{x})u+b_{2}^{\prime }g\left(
u,\eta _{x}\right) g\left( u,N_{a}\right) u\right] - \\
g\left[ V,\frac{b_{2}}{2}g(u,\eta _{x})\delta _{a}\right] ,
\end{multline*}

\item 
\begin{multline*}
\widetilde{G}\left[ H^{h}+V^{v},\widetilde{\nabla }_{\partial _{a}}\partial
_{b}\right] =\widetilde{G}\left[ H^{h}+V^{v},\widetilde{H}(\partial
_{a},\partial _{b})\right] = \\
g[AH+a_{2}V,\nabla _{\delta _{a}}\delta _{b}]-a_{2}R(H,\delta _{b},u,\delta
_{a})- \\
\frac{a_{1}}{2}\left[ R(H,\delta _{a},u,M_{b}+N_{b})+R(H,\delta
_{b},u,M_{a}+N_{a})+R(u,V,\delta _{a},\delta _{b}\right] + \\
g\left[ a_{2}H+a_{1}V,\nabla _{\delta _{a}}(M_{b}+N_{b})\right] - \\
g(u,V)\left[ A^{\prime }g(\delta _{a},\delta _{b})+a_{1}^{\prime }\left(
g(M_{a},M_{b})+g(N_{a},N_{b})\right) +2a_{2}g(M_{a},\delta _{b})\right] - \\
\frac{1}{2}b_{2}\left[ g(u,N_{a})g(V,\delta _{b})+g(u,N_{b})g(V,\delta _{a})%
\right] + \\
g\left[ H,\frac{b_{2}}{2}g(u,N_{a})\left( M_{b}+N_{b}\right) +\frac{b_{2}}{2}%
g(u,N_{b})\left( M_{a}+N_{a}\right) \right] + \\
b_{2}^{\prime }g(u,H)g(u,N_{a})g(u,N_{b})
\end{multline*}
\end{enumerate}

for a nondegenerate $g$-natural metric $\widetilde{G}.$
\end{theorem}

\begin{proof}
Straightforward computation.
\end{proof}

\begin{remark}
The third equation of the last Proposition can be written as%
\begin{multline*}
\widetilde{G}\left[ H^{h}+V^{v},\widetilde{\nabla }_{\partial _{a}}\partial
_{b}\right] =\widetilde{G}\left[ H^{h}+V^{v},\widetilde{H}(\partial
_{a},\partial _{b})\right] = \\
G[H^{h}+V^{v},(\nabla _{\delta _{a}}\delta _{b})^{h}+(\nabla _{\delta
_{a}}\nabla _{\delta _{b}}u)^{v}]-g(BH+b_{2}V,u)g(u,\nabla _{\delta
_{a}}\nabla _{\delta _{b}}u)- \\
a_{2}R(H,\delta _{b},u,\delta _{a})-\frac{a_{1}}{2}\left[ R(H,\delta
_{a},u,M_{b}+N_{b})+R(H,\delta _{b},u,M_{a}+N_{a})+R(u,V,\delta _{a},\delta
_{b}\right] + \\
g(u,V)\left[ A^{\prime }g(\delta _{a},\delta _{b})+a_{1}^{\prime }\left(
g(M_{a},M_{b})+g(N_{a},N_{b})\right) +2a_{2}g(M_{a},\delta _{b})\right] - \\
\frac{1}{2}b_{2}\left[ g(u,N_{a})g(V,\delta _{b})+g(u,N_{b})g(V,\delta _{a})%
\right] + \\
g\left[ H,\frac{b_{2}}{2}g(u,N_{a})\left( M_{b}+N_{b}\right) +\frac{b_{2}}{2}%
g(u,N_{b})\left( M_{a}+N_{a}\right) \right] + \\
b_{2}^{\prime }g(u,H)g(u,N_{a})g(u,N_{b}).
\end{multline*}
\end{remark}

\begin{definition}
A distribution $D$ on a manifold $M$ is said to be totally geodesic if it is
invariant with respect to covariant differentiation, i.e. $\nabla _{X}Y\in D$
for all $X,Y\in D.$
\end{definition}

\begin{theorem}
\label{th 14}\label{Prop 9}Let $M,$ $\func{codim}M>1,$ be a submanifold
isometrically immersed in a manifold $(N,g).$ Suppose that $LM$ is a
submanifold isometrically immersed by (\ref{imersion}) in $TN$ with
non-degenerate $g$-natural metric $G.$

\begin{enumerate}
\item If either the normal bundle of $LM$ is spanned by vectors of the form $%
H^{h}+V^{v},$ where $H$ and $V$ are tangent to $M$ or $b_{2}=a_{2}^{\prime
}=0$ along $M,$ then vector fields $\{\partial _{x}\},$ $x=m+1,...,n$ define
on $LM$ the totally geodesic distribution that is involutive.

\item If

\begin{enumerate}
\item $a_{1}=0,$ $b_{2}=0,$ $a_{2}=const\neq 0$ along $M$ or

\item $N$ is a space of constant curvature and $a_{2}^{\prime }=0,$ $b_{2}=0$
along $M$,
\end{enumerate}
\end{enumerate}

then $LM$ is mixed totally geodesic.

Here, along $M,$%
\begin{equation*}
\nabla _{\delta _{a}}u=\nabla _{\delta _{a}}(v^{y}\eta _{y})=\nabla _{\delta
_{a}}\left( v^{y}N_{y}^{r}\frac{\partial }{\partial x^{r}}\right)
=M_{a}+N_{a}.
\end{equation*}
\end{theorem}

\begin{proof}
In virtue of the assumptions, the first equation of Theorem \ref{th 10}
yields $\widetilde{G}\left[ H^{h}+V^{v},\widetilde{H}(\partial _{x},\partial
_{y})\right] =0.$ Hence $\widetilde{\nabla }_{\partial _{x}}\partial _{y}$
is tangent to $LM.$ Since $\partial _{x}=\eta _{x}^{v}$ are vertical vector
fields, the distribution is involutive. This proves the first point. The
proof of the second one is obvious.
\end{proof}

Stanis\l aw Ewert-Krzemieniewski

West Pomeranian University of Technology Szczecin

School of Mathematics

Al. Piast\'{o}w 17, 70-310 Szczecin, Poland

e-mail: ewert@zut.edu.pl

\end{document}